\documentclass{amsart}
\usepackage{amsmath,amsthm,amssymb,amscd}

\newtheorem{thm}{Theorem}[section]
\newtheorem{prop}[thm]{Proposition}
\newtheorem{lem}[thm]{Lemma}
\newtheorem{cor}[thm]{Corollary}

\theoremstyle{definition}
\newtheorem{dfn}[thm]{Definition}
\theoremstyle{remark}
\newtheorem{eg}[thm]{Example}
\theoremstyle{remark}
\newtheorem{rem}[thm]{Remark}

\newcommand{\Hom}{{\rm Hom}}
\newcommand{\Sing}{{\rm Sing}}

\newcommand{\R}{\mathbb{R}}

\newcommand{\J}{\mathcal{J}}
\newcommand{\Cliq}{{\rm Cliq}}

\title{Morphism complexes of sets with relations}
\author{Takahiro Matsushita}
\email{tmatsu@ms.u-tokyo.ac.jp}
\address{Graduate School of the Mathematical Sciences, the University of Tokyo, 3-8-1, Komaba, Meguro-ku, Tokyo, Japan}

\begin{document}

\maketitle

\begin{abstract}
Let $r$ be a positive integer. An $r$-set is a pair $X= (V(X),R(X))$ consisting of  a set $V(X)$ with a subset $R(X)$ of the direct product $V(X)^r$. The object of this paper is to investigate the Hom complexes of $r$-sets, which were introduced for graphs in the context of the graph coloring problem.

In the first part, we introduce simplicial sets which we call singular complexes, and show that singular complexes and Hom complexes are naturally homotopy equivalent. The second part is devoted to the generalization of $\times$-homotopy theory established by Dochtermann. We show the folding theorem for hypergraphs which was partly proved by Iriye and Kishimoto.
\end{abstract}

\section{Introduction}
One of the most remarkable applications of algebraic topology to combinatorics is Lov$\acute{\rm a}$sz's proof of Kneser's conjecture \cite{Lov}. He assigned a simplicial complex to a graph, and related its connectivity to the chromatic number. Hom complex was also introduced by Lov$\acute{\rm a}$sz in the context of the graph coloring problem, and was later developed by Babson and Kozlov in \cite{BK1} and \cite{BK2}.

The object of the paper is to investigate Hom complexes. As was mentioned in \cite{Koz}, Hom complexes can be defined not only for graphs but for more general objects. In fact, Hom complexes of $r$-uniform hypergraphs were recently considered in \cite{IK} and \cite{Tha}. Thus we consider Hom complexes of more general objects, namely, $r$-{\it sets}.

Throughout this paper, $r$ shall denote a fixed positive integer. An $r$-set is a pair $X=(V(X),R(X))$ consisting of a set $V(X)$ with a subset $R(X)$ of the $r$-times direct product of $V(X)$. We call $V(X)$ the vertex set of $X$ and $R(X)$ the $r$-relation of $X$. We note that $V(X)$ may be infinite. 

Let $\mathfrak{S}_r$ denote the symmetric group on the set $\{ 1,\cdots, r\}$. An $r$-uniform hypergraph is an $r$-set $X$ whose $r$-relation is closed under the $\mathfrak{S}_r$-action on $V(X)^r$ by permutation. Therefore an $r$-set is a generalization of an $r$-uniform hypergraph.

As is the case of graphs, we define the Hom complex of $r$-sets in the following way. A map $f: V(X) \rightarrow V(Y)$ is a {\it homomorphism} if $f^{\times r} (R(X)) \subset R(Y)$. A map $\eta : V(X) \rightarrow 2^{V(Y)} \setminus \{ \emptyset\}$ is a {\it multi-homomorphism} if
$$\eta(x_1) \times \cdots \times \eta(x_r) \subset R(Y)$$
for every element $(x_1,\cdots, x_r)$ of $R(X)$. For two multi-homomorphisms $\eta$ and $\eta'$, we write $\eta \leq \eta'$ to indicate that $\eta(v) \subset \eta'(v)$ for all $v \in V(X)$. The {\it Hom complex} $\Hom(X,Y)$ is the poset consisting of multi-homomorphisms from $X$ to $Y$ together with the ordering mentioned above.

The contents of this paper are divided into two parts. In the first part, we construct the simplicial set $\Sing(X,Y)$ which we call {\it singular complex} and relate it to the Hom complex. To give the precise definition of singular complex, we need some preparation.

Let ${\bf Set}_r$ denote the category of $r$-sets whose morphisms are homomorphisms. It will be shown in Section 3 that ${\bf Set}_r$ admits all small limits and colimits. For instance, the product $r$-set $X \times Y$ of two $r$-sets $X$ and $Y$ is defined by
$$V(X \times Y) = V(X) \times V(Y),$$
$$R(X \times Y) = \{ ((x_1,y_1), \cdots, (x_r,y_r)) \; | \; (x_1,\cdots,x_r) \in R(X), (y_1,\cdots,y_r) \in R(Y)\}.$$

For a non-negative integer $n$, we define the $r$-set $\Sigma_n$ by $V(\Sigma_n) = \{ 0,1,\cdots,n\}$ and $R(\Sigma_n) = V(\Sigma_n)^r$. The {\it singular complex} $\Sing(X,Y)$ which one associates to a pair $(X,Y)$ of $r$-sets is the simplicial set defined by
$$\Sing(X,Y)_n = \{ f:X \times \Sigma_n \rightarrow Y \; | \; f \textrm{ is a homomorphism.}\}.$$
In terms of these notions, our principal result is formulated as follows.

\vspace{2mm} \noindent {\bf Theorem 4.1.}
{\it There is a natural homotopy equivalence}
$$\begin{CD}
|\Sing (X,Y)| @>{\simeq}>> |\Hom(X,Y)|.
\end{CD}$$

\vspace{2mm}
Theorem 4.1 gives another description of the homology groups of Hom complexes. Let $C_n(X,Y)$ denote the free abelian group generated by the set of homomorphisms from $X \times \Sigma_n$ to $Y$. The differential $\partial : C_n(X,Y) \rightarrow C_{n-1}(X,Y)$ is obviously defined. Theorem 4.1 implies that the homology group of the complex $C_\bullet (X,Y)$ is isomorphic to the homology group (with integral coefficients) of $\Hom(X,Y)$. This description is similar to the singular homology group of a topological space. This is why we call the simplicial set $\Sing(X,Y)$ the singular complex.

Let ${\bf SSet}$ denote the category of simplicial sets. We note that for an $r$-set $X$, the functor ${\bf Set}_r \rightarrow {\bf SSet}$, $Y \mapsto \Sing(X,Y)$ is an associated functor of the cosimplicial $r$-set $[n] \mapsto X \times \Sigma_n$. As was mentioned, the category ${\bf Set}_r$ of $r$-sets admits all small colimits. Because of this and the well-known fact of simplicial sets (Theorem 2.3), the functor $Y \mapsto \Sing(X,Y)$ has the left adjoint.

The object of the second part which we discuss in Section 5 is to generalize the $\times$-homotopy theory of graphs introduced by Dochtermann \cite{Doc} to $r$-sets. We relate the $\times$-homotopy theory to the homotopy theory of posets and strong homotopy theory of finite simplicial complexes \cite{BM}.

We note that a homomorphism $f:X \rightarrow Y$ between $r$-sets is identified with the multi-homomorphism $x \mapsto \{ f(x)\}$. Recall that two continuous maps $\varphi$ and $\psi$ between two (compactly generated) spaces are homotopic if and only if there is a path joining $\varphi$ to $\psi$ on the function space. From this viewpoint the following definition is quite natural. Two homomorphisms $f,g:X \rightarrow Y$ are {\it strongly homotopic} if $f$ and $g$ belong to the same connected component of $\Hom(X,Y)$.

A vertex $x$ of the $r$-set $X$ is {\it dismantlable} if there is another vertex $y$ of $X$ such that $p_i^{-1}(x) \subset p_i^{-1}(y)$ for $i=1,\cdots, r$, where $p_i : R(X) \rightarrow V(X)$ is the $i$-th projection. Let $X \setminus x$ denote the maximal $r$-subset of $X$ whose vertex set is $V(X) \setminus \{ x\}$. As an application of strong homotopy theory of $r$-sets, we have that if $x$ is a dismantlable vertex of $X$, then the maps
$$i^*:\Hom(X, Y) \mapsto \Hom(X\setminus x,Y)$$
and
$$i_*:\Hom(Y,X \setminus x) \rightarrow \Hom(Y,X)$$
are homotopy equivalences (Theorem 5.6). In the case of graphs, Babson and Kozlov showed that $i^*$ is a homotopy equivalence (Proposition 5.1 of \cite{BK1}), and Kozlov later showed that $i_*$ is a homotopy equivalence \cite{Koz2}. Iriye and Kishimoto showed that $i^*$ is a homotopy equivalence for uniform hypergraphs (Theorem 17 of \cite{IK}). The part $i_*$ is a homotopy equivalence for uniform hypergraphs is a new result.

The strong homotopy type of an $r$-set is determined by its {\it weak core} (Theorem 5.15). A weak core is a homomorphism $i:X' \rightarrow X$ where $i$ is a strong homotopy equivalence and $X'$ has no dismantlable vertices.

We conclude this section by mentioning our terminology. An $r$-uniform hypergraph $X$ is {\it non-degenerate} if for each element $(x_1,\cdots, x_r)$ of $R(X)$, $x_1,\cdots, x_r$ are distinct. In some literature ``$r$-uniform hypergraph" means non-degenerate $r$-uniform hypergraph in our sense. One of the reasons why we employ such terminology is that, as was mentioned in \cite{IK}, we need to admit degeneracies to apply the Hom complexes to the hypergraph coloring problem. The second reason is that the category of non-degenerate uniform hypergraphs does not admit small limits and colimits (see Remark 3.6).

\vspace{2mm}
\noindent {\bf Acknowledgement.} The author is grateful to his supervisor Toshitake Kohno for his enormous support and helpful suggestions. The author expresses his appreciation to Professor Dai Tamaki and Shouta Tounai for their helpful and insightful comments. He also thanks to the anonymous referee for helpful suggestions which improve the paper. The author is supported by the Grant-in-Aid for Scientific Research (KAKENHI No. 25-4699) and the Grant-in-Aid for JSPS fellows. This work was supported by the Program for Leading Graduate Schools, MEXT, Japan.

\section{Preliminaries}
In this section, we review definitions and some properties of abstract simplicial complexes, posets, and simplicial sets following \cite{Hov}, \cite{Koz}, and \cite{May}.

\subsection{Simplicial complex and poset} An {\it (abstract) simplicial complex} is a pair $(V,K)$ consisting of a set $V$ together with a collection of finite subsets of $V$ such that if $\sigma \in K$ and $\tau \subset \sigma$ then $\tau \in K$. Furthermore, we require that $v \in V$ implies $\{ v\} \in K$. The set $V$ is the {\it vertex set} of the simplicial complex $(V,K)$. We call an element of $K$ a {\it simplex}. A simplicial complex is often denoted simply by $K$. In this notation, we write $V(K)$ to indicate the vertex set of $K$. A map $f:V(K) \rightarrow V(K')$ is a {\it simplicial map} if $f(\sigma) \in K'$ for every $\sigma \in K$.

Let $V$ be a set and let $\R^{(V)}$ be the free $\R$-module generated by $V$. We regard $\R^{(V)}$ as a topological space whose topology is induced by finite dimensional $\R$-submodules, and regard an element of $V$ as a point of $\R^{(V)}$ in the usual way. The {\it geometric realization} of the simplicial complex $K$ is the union of the convex hulls in $\R^{(V(K))}$
$$|K| = \bigcup_{\sigma \in K} {\rm conv}(\sigma)$$
of simplices of $K$.

Let $P$ be a partially ordered set (poset, for short). A subset $c$ of $P$ is a {\it chain} in $P$ if the restriction of the ordering of $P$ to $c$ is a total ordering. The order complex of $P$, denoted by $\Delta(P)$, is the simplicial complex whose vertices are elements of $P$ and whose simplices are finite chains in $P$. We write $|P|$ instead of $|\Delta (P)|$, and call it the geometric realization of $P$.

The geometric realization functor allows us to assign topological concepts to posets and simplicial complexes. For example, we call two order preserving maps $f$ and $g$ homotopic if $|f|$ and $|g|$ are homotopic.

Let $K$ be a simplicial complex. The face poset $FK$ of $K$ is the poset of non-empty simplices of $K$ by inclusion. The {\it barycentric subdivision} of $K$ is the order complex of $FK$.

\begin{thm}
There is a natural homeomorphism
$$\begin{CD}
|FK| @>{\cong}>>|K|.
\end{CD}$$
\end{thm}

\begin{thm}[Quillen \cite{Qui}]
Let $f:P \rightarrow Q$ be an order preserving map. If $f^{-1}(Q_{\leq y})$ is contractible for all $y \in Q$, then $f$ is a homotopy equivalence.
\end{thm}

\subsection{Strong homotopy of posets} Let $P$, $Q$ be posets, and let $f,g : P \rightarrow Q$ be order preserving maps. We write $f\leq g$ to indicate that $f(x) \leq g(x)$ for every element $x$ of $P$. Let ${\rm Poset}(P,Q)$ denote the poset consisting of all order preserving maps from $P$ to $Q$ together with the above ordering. Let $P \times Q$ denote the categorical product of posets $P$, $Q$. Namely, the underlying set of $P \times Q$ is the cartesian product of their underlying sets, and the ordering is given by that $(x,y) \leq (x',y')$ if and only if $x \leq x'$ and $y \leq y'$. Then one can verify that there is a natural isomorphism
$${\rm Poset}(P, {\rm Poset}(Q,R)) \cong {\rm Poset} (P \times Q,R)$$
for posets $P$, $Q$, and $R$.

Order preserving maps $f,g : P \rightarrow Q$ are {\it strongly homotopic} if $f$ and $g$ belong to the same connected component of $|{\rm Poset}(P,Q)|$. We write $f \sim_s g$ to mean that $f$ and $g$ are strongly homotopic. It is known that if $f\leq g$, then $f$ and $g$ are homotopic. Hence if $f$ and $g$ are strongly homotopic, then they are homotopic.

An order preserving map $f:P \rightarrow Q$ is a {\it strong equivalence} if there is an order preserving map $g: Q \rightarrow P$ such that $gf \sim_s {\rm id}_P$ and $fg \sim_s {\rm id}_Q$.

The terminology ``strongly homotopic" and ``strong equivalence" are not standard. However, these notions have been known in terms of finite spaces \cite{Bar}. Recall that the category of finite posets and the category of finite $T_0$-spaces are equivalent. From this viewpoint, two order preserving maps $f$ and $g$ are strongly homotopic in our sense if and only if continuous maps associated to $f$ and $g$ are homotopic. The reason why we use such terminology is that a strong equivalence of posets is closely related to strong equivalence of finite simplicial complexes introduced by Barmak and Minian \cite{BM}. For instance, Barmak and Minian show that if an order preserving map $f:P \rightarrow Q$ between finite posets is a strong equivalence, then the associated simplicial map $\Delta(f):\Delta(P) \rightarrow \Delta (Q)$ is a strong equivalence of finite simplicial complexes. Since a strong equivalence between finite simplicial complexes is a simple homotopy equivalence (Proposition 2.5 of \cite{BM}), a simplicial map $\Delta(f):\Delta(P) \rightarrow \Delta(Q)$ associated to a strong equivalence $f$ between finite posets is a simple homotopy equivalence.

\subsection{Simplicial set} For a non-negative integer $n$, we write $[n]$ to mean the linearly ordered set $\{ 0,1,\cdots, n\}$. Let ${\bf \Delta}$ be the small category whose objects are $[n]$ for $n \geq 0$ and whose morphisms are order preserving maps. A {\it simplicial set} is a functor from the opposite category of ${\bf \Delta}$ to the category of sets. Morphisms between two simplicial sets are defined by natural transformations. Let ${\bf SSet}$ denote the category of simplicial sets. For a simplicial set $K$, we write $K_n$ instead of $K [n]$.

The canonical $n$-simplex $\Delta^n$ is the subspace of $\mathbb{R}^{n+1}$ defined by
$$\Delta^n = \Big\{ x_0 e_0 + \cdots + x_n e_n \; | \; x_i \geq 0 \; (i=0,1,\cdots, n), \; \sum_{i=0}^n x_i =1\Big\}$$
where $e_0,\cdots, e_n$ are the canonical basis of $\mathbb{R}^{n+1}$.

The geometric realization of a simplicial set $K$ is defined as follows. First we assign a canonical $n$-simplex $\Delta(\sigma)$ to each element $\sigma$ of $K_n$. The {\it geometric realization} of the simplicial set $K$ is the quotient space
$$\coprod_{n \geq 0, \sigma \in K_n} \Delta(\sigma)/ \sim$$
where the equivalence relation $\sim$ is generated by the relation
$$\Delta(f^* \sigma) \ni x_0e_0 + \cdots + x_n e_n \sim x_0 e_{f(0)} + \cdots + x_n e_{f(n)} \in \Delta (\sigma), \; (f:[n] \rightarrow [m]).$$

Let $\mathcal{C}$ be a category. A {\it cosimplicial object} of the category $\mathcal{C}$ is a functor from ${\bf \Delta}$ to $\mathcal{C}$. Let $A^\bullet : {\bf \Delta} \rightarrow \mathcal{C}$ be a cosimplicial object of $\mathcal{C}$. The functor $\mathcal{C}(A^\bullet,-) : \mathcal{C} \rightarrow {\bf SSet}$ is defined by $\mathcal{C}(A^\bullet, X)_n = \mathcal{C}(A^n, X)$.

\begin{thm}[Proposition 3.1.5 of \cite{Hov}]
If the category $\mathcal{C}$ admits all small colimits, then the functor $\mathcal{C}(A^\bullet, -)$ has the left adjoint.
\end{thm}

\subsection{Gluing lemma}

We will need the following theorem in Section 4.

\begin{thm}
Let $X$ and $Y$ be CW-complexes and let $A$ be a set. Let $(X_\alpha)_{\alpha \in A}$ (or $(Y_{\alpha})_{\alpha \in A}$) be an $A$-indexed family of subcomplexes of $X$ (or $Y$) which is a covering of $X$ (or $Y$ respectively). Let $f: X \rightarrow Y$ be a continuous map $X \rightarrow Y$ such that $f(X_\alpha) \subset Y_\alpha$ for all $\alpha \in A$. Suppose that for any finite subset $\sigma \in A$, the map
$$f|_{\bigcap_{\alpha \in \sigma} X_\alpha} : \bigcap_{\alpha \in \sigma} X_\alpha \rightarrow \bigcap_{\alpha \in \sigma} Y_\alpha$$
is a homotopy equivalence. Then the map $f$ is a homotopy equivalence.
\end{thm}
\begin{proof}
This theorem is well-known if $A$ is finite (see Section 15.5.1 of \cite{Koz}). Hence we only deal with the infinite case. First we introduce the notation. For a subset $\sigma$ of $A$, we write $X_\sigma$ (or $Y_\sigma$) to indicate the union $\bigcup_{\alpha \in \sigma} X_\alpha$ (or $\bigcup_{\alpha \in \sigma} Y_\alpha$ respectively). It follows from the finite case that $X_\sigma \rightarrow Y_\sigma$ is a homotopy equivalence.

We can assume that $X$ and $Y$ are non-empty. Let $x \in X$. It is enough to show that $\pi_n(X,x) \rightarrow \pi_n(Y,f(x))$ is bijective for $n \geq 0$. Let $\varphi :(S^n, *) \rightarrow (Y,f(x))$ be a continuous map. Since $\varphi (S^n)$ is compact there is a finite subset $\sigma \subset A$ such that
$$x \in X_\sigma, \varphi (S^n) \subset Y_\sigma.$$
Since $f|_{X_\sigma *} : \pi_n(X_\sigma,x) \rightarrow \pi_n(Y_\sigma,f(x))$ is bijective, there is $\psi : (S^n,*) \rightarrow (X_\sigma, x)$ such that $f \circ \psi \simeq \varphi$. This implies that $f_* : \pi_n(X,x) \rightarrow \pi_n(Y,f(x))$ is surjective. The injectivity of $f_*$ is similarly obtained.
\end{proof}

\section{Limits and colimits}

Let ${\bf Set}_r$ be the category of $r$-sets and let ${\bf Graph}_r$ be the category of $r$-uniform hypergraphs. The aim of this section is to show that ${\bf Set}_r$ and ${\bf Graph}_r$ admit all small limits and colimits. First we deal with the case of $r$-sets.

Throughout this section $\mathcal{J}$ shall denote a small category. We typically write $j \in \mathcal{J}$ to indicate that $j$ is an object of $\mathcal{J}$. Let $\varphi :\mathcal{J} \rightarrow {\bf Set}_r$ be a functor. The limit $\lim (\varphi) \in {\bf Set}_r$ of $\varphi$ is defined by
$$V(\lim (\varphi)) = \Big\{ (x_j)_{j \in \J} \in \prod_{j \in \J} V(\varphi(j)) \; | \; \varphi(u)(x_{j_0}) = x_{j_1} \textrm{ for a morphism $u: j_0 \rightarrow j_1$ in }\J.\Big\},$$
$$R(\lim (\varphi)) = \{ ((x_j^1)_{j \in \J},\cdots, (x_j^r)_{j \in \J}) \; | \; (x_j^1,\cdots,x_j^r) \in R(\varphi(j)) \textrm{ for }j \in \J.\}.$$

Let $J$ be a small set and let $(X_j)_{j \in J}$ be a $J$-indexed family of $r$-sets. The coproduct $\coprod_{j \in J}X_j$ is defined by
$$V \Big( \coprod_{j \in J} X_j \Big) = \coprod_{j \in J} V(X_j),$$
$$R \Big( \coprod_{j \in J} X_j \Big) = \coprod_{j \in J} R(X_j).$$

Let $X$ be an $r$-set and let $\sim$ be an equivalence relation on $V(X)$. The quotient $r$-set $X/\sim$ is defined by
$$V(X/\sim) = V(X)/\sim,$$
$$R(X/\sim) = \{ (\alpha_1,\cdots, \alpha_r) \; | \; (\alpha_1 \times \cdots \times \alpha_r) \cap R(X) \neq \emptyset.\}.$$
Then the quotient map $\pi : V(X) \rightarrow V(X/\sim)$ is a homomorphism. Furthermore, this has the following universality.

\begin{lem}
Let $f:X \rightarrow Y$ be a homomorphism such that if $x \sim y$ then $f(x) = f(y)$. Then there is a unique homomorphism $\overline{f}: (X/\sim) \rightarrow Y$ satisfying $\overline{f} \circ \pi = f$.
\end{lem}
\begin{proof}
It suffices to show that the set map $\overline{f}$ induced by the set map $f: V(X) \rightarrow V(Y)$ is a homomorphism. Let $(\alpha_1,\cdots, \alpha_r) \in R(X/\sim)$ and let $x_i \in \alpha _i$ $(i=1,\cdots, r)$ such that $(x_1,\cdots, x_r) \in R(X)$. Then we have
$$\overline{f}(\alpha_1,\cdots, \alpha_r) = f(x_1,\cdots, x_r) \in R(Y).$$
Therefore the map $\overline{f}$ is a homomorphism of $r$-sets.
\end{proof}

Let us construct the colimit of the functor $\varphi: \mathcal{J} \rightarrow {\bf Set}_r$. Let $\sim_\varphi$ denote the equivalence relation on the vertex set of the coproduct $\coprod_{j \in \J} \varphi (j)$ generated by the relations: $x \sim_\varphi \varphi(u)(x)$ for $x \in \varphi(j_0)$ and a morphism $u:j_0 \rightarrow j_1$ in $\J$. Then the colimit of $\varphi$ is defined by
$${\rm colim}(\varphi) = \coprod_{j \in \J} \varphi(j) / \sim_\varphi.$$

\begin{thm}
The category ${\bf Set}_r$ of $r$-sets admits all small limits and colimits.
\end{thm}

Next we deal with the category ${\bf Graph}_r$ of $r$-uniform hypergraphs.

\begin{dfn}
Let $X$ be an $r$-set.
\begin{itemize}
\item[(1)] Let $FX$ denote the $r$-uniform hypergraph defined by $V(FX) = V(X)$ and
$$R(FX) = \{ (x_1,\cdots, x_r) \; | \; \textrm{There is $\sigma \in \mathfrak{S}_r$ such that }(x_{\sigma(1)},\cdots ,x_{\sigma(r)}) \in R(X).\}$$
\item[(2)] Let $UX$ denote the $r$-uniform hypergraph defined by $V(UX) = V(X)$ and
$$R(UX) = \{ (x_1,\cdots,x_r) \; | \; (x_{\sigma(1)},\cdots, x_{\sigma(r)}) \in R(X)\textrm{ for every }\sigma \in \mathfrak{S}_r.\} $$
\end{itemize}
\end{dfn}

For a homomorphism $f:X \rightarrow Y$ of $r$-sets, we put $Ff = Uf = f$. Then $F$ and $U$ are functors from ${\bf Set}_r$ to ${\bf Graphs}_r$. Let $\iota$ denote the inclusion functor ${\bf Graph}_r \hookrightarrow {\bf Set}_r$. Then we have the following proposition.

\begin{prop}
The functor $F$ is the left adjoint of $\iota$ and the functor $U$ is the right adjoint of $\iota$.
\end{prop}
\begin{proof}
Let $X$ be an $r$-set and let $Y$ be an $r$-uniform hypergraph. Let $f:X \rightarrow Y$ be a homomorphism. We want to show that $f:V(FX) = V(X) \rightarrow V(Y)$ is again a homomorphism from $FX$ to $Y$. Let $(x_1,\cdots,x_r) \in R(FX)$. Then there is $\sigma \in \mathfrak{S}_r$ such that $(x_{\sigma(1)},\cdots, x_{\sigma(r)}) \in R(X)$. We have $(f(x_{\sigma(1)}),\cdots, f(x_{\sigma(r)})) \in R(Y)$ since $f$ is a homomorphism. Since $Y$ is an $r$-uniform hypergraph, we have $(f(x_1),\cdots, f(x_r)) \in R(Y)$.

Next let $g$ be a homomorphism from $Y$ to $X$. We want to show that $g:V(Y) \rightarrow V(X) = V(UX)$ is a homomorphism from $Y$ to $UX$. Let $(y_1,\cdots,y_r)\in R(Y)$. Then we have that $(y_{\sigma(1)},\cdots, y_{\sigma(r)}) \in R(Y)$ for each $\sigma \in \mathfrak{S}_r$. Hence we have $(f(y_{\sigma(1}),\cdots, f(y_{\sigma(r)})) \in R(X)$ for each $\sigma \in \mathfrak{S}_r$. This implies that $(f(y_1),\cdots, f(y_r)) \in R(X)$.
\end{proof}

\begin{cor}
The category ${\bf Graph}_r$ of $r$-uniform hypergraphs admits all small limits and colimits.
\end{cor}
\begin{proof}
Let $\varphi:\J \rightarrow {\bf Graphs}_r$ be a functor. For each $r$-uniform hypergraph $X$, we have
\begin{eqnarray*}
{\bf Graphs}_r (X, U(\lim \iota \varphi)) &\cong & {\bf Set}_r (\iota X, \lim \iota \varphi)\\
&\cong & \lim_{j} ({\bf Set}_r(\iota X, \iota \varphi(j)))\\
&\cong & \lim_{j} ({\bf Graphs}_r (X, \varphi (j))).
\end{eqnarray*}
This implies that $U(\lim (\iota \varphi ))$ is the limit of $\varphi$. We have similarly that $F({\rm colim} (\iota \varphi ))$ is the colimit of $\varphi$.
\end{proof}

\begin{rem}
As was mentioned in Section 1, an $r$-uniform hypergraph $X$ is non-degenerate if for each element $(x_1,\cdots, x_r)$ of $R(X)$, $x_1,\cdots, x_r$ are distinct. Let $\mathcal{G}$ be the full subcategory of the category ${\bf Graph}_r$ consisting of non-degenerate $r$-uniform hypergraphs. Then $\mathcal{G}$ does not admit finite limits and finite colimits.

In fact $\mathcal{G}$ does not admit finite limits since $\mathcal{G}$ does not have the terminal object. On the other hand, let $K_r$ be the $r$-uniform hypergraph defined by $V(K_r) = \{ 1,\cdots, r\}$ and $R(K_r) = \{ (x_1,\cdots, x_r) \; | \; x_i \neq x_j \; (i \neq j)\}$. We regard the symmetric group $\mathfrak{S}_r$ as a small category in a usual way. Namely, the object of $\mathfrak{S}_r$ is only one $*$ and the morphism set from $*$ to $*$ is the group $\mathfrak{S}_r$. Let $\varphi : \mathfrak{S}_r \rightarrow \mathcal{G}$ be the functor defined by $\varphi (*) = K_r$ and $\varphi (\sigma) (x) = \sigma(x)$ for $\sigma \in \mathfrak{S}_r$. This functor does not have the colimit.
\end{rem}

\section{Singular complex}

Recall that the singular complex is defined by the right adjoint functor $\Sing(X,-) : {\bf Set}_r \rightarrow {\bf SSet}$ associated to the cosimplicial $r$-set
$${\bf \Delta} \rightarrow {\bf Set}_r, [n] \mapsto X \times \Sigma_n$$
for an $r$-set $X$. Namely, the singular complex $\Sing(X,Y)$ is the simplicial set
$$\Sing(X,Y)_n = \{ f:X \times \Sigma_n \rightarrow Y \; | \; f\textrm{ is a homomorphism of $r$-sets.}\}$$
with obvious face maps and degeneracy maps. The aim of this section is to show the following theorem.

\begin{thm}
There is a natural homotopy equivalence
$$\begin{CD}
|\Sing(X,Y)| @>{\simeq}>> |\Hom(X,Y)|.
\end{CD}$$
\end{thm}

Let $X$ and $Y$ be $r$-sets. A multi-homomorphism $\eta \in \Hom(X,Y)$ is finite if $\eta (x)$ is finite for each $x \in V(X)$. The induced subposet of $\Hom(X,Y)$ consisting of all finite multi-homomorphisms is denoted by $\Hom_f(X,Y)$.

If $X$ and $Y$ are finite $r$-sets then $\Hom_f(X,Y) = \Hom(X,Y)$. In general the inclusion $\Hom_f(X,Y) \hookrightarrow \Hom(X,Y)$ is a homotopy equivalence. This fact is deduced from the following lemma and Quillen's Theorem A (Theorem 2.2).

\begin{lem}
Let $P$ be a poset. If there is an upper bound for every finite subset of $P$, then $P$ is contractible.
\end{lem}
\begin{proof}
Since the empty subset has an upper bound, $P$ is not empty. By the hypothesis, every finite subcomplex of $\Delta(P)$ is included in a certain contractible subcomplex. This implies that a map from a sphere to $|\Delta (P)|$ is null-homotopic, and hence $P$ is contractible by the Whitehead theorem.
\end{proof}

\begin{dfn}
Let $X$ and $Y$ be $r$-sets. The {\it morphism $r$-set} $Y^X$ is defined by
$$V(Y^X) = \{ f:V(X) \rightarrow V(Y) \; | \; f \textrm{ is a map of sets.}\},$$
$$R(Y^X) = \{ (f_1,\cdots, f_r) \; | \; (f_1 \times \cdots \times f_r)(R(X)) \subset R(Y)\}.$$
\end{dfn}

It can be verified that there is a natural isomorphism ${\bf Set}_r(X \times Y,Z) \cong {\bf Set}_r(X, Z^Y)$. Hence we have $\Sing(X \times Y,Z) \cong \Sing(X,Z^Y)$. On the other hand the following holds in the case of Hom complexes.

\begin{lem}
There is a natural strong equivalence
$$\Hom(X,Z^Y) \rightarrow \Hom(X \times Y,Z).$$
\end{lem}
\begin{proof}
Dochtermann proved this lemma in the case of graphs (Proposition 3.5 of \cite{Doc}) although he did not use the term ``strong equivalence." A similar proof works well. Hence we give a sketch of the proof.

The maps $\Phi : \Hom (X \times Y, Z) \rightarrow \Hom (X,Z^Y)$ and $\Psi : \Hom (X,Z^Y) \rightarrow \Hom(X \times Y,Z)$ are defined by
$$\Phi (\eta)(x) = \{ f:V(Y) \rightarrow V(Z) \; | \; f(y) \in \eta(x,y) \textrm{ for }y \in V(Y).\},$$
$$\Psi(\eta) (x,y) = \{ f(y) \; | \; f \in \eta (x)\}.$$
Then one can show that $\Psi \circ \Phi = {\rm id}$, and $\Phi \circ \Psi \geq {\rm id}$.
\end{proof}

Since $\Sigma_0 \times X \cong X$, we have that $\Sing(X,Y) \cong \Sing(\Sigma_0, Y^X)$ and $\Hom(X,Y) \simeq \Hom(\Sigma_0,Y^X) \simeq \Hom_f(\Sigma_0, Y^X)$. Hence it suffices to construct a homotopy equivalence $|\Sing(\Sigma_0,X)| \rightarrow |\Hom_f(\Sigma_0, X)|$.

A subset $A$ of an $r$-set $X$ is a {\it clique} if $A^r$ is included in $R(X)$. The {\it clique complex} $\Cliq (X)$ is the simplicial complex whose simplices are finite cliques of $X$. Since $\Hom_f(\Sigma_0,X)$ is isomorphic to the face poset of $\Cliq (X)$, there is a homeomorphism $|\Cliq(X)| \rightarrow |\Hom_f(\Sigma_0,X)|$.

We write $\Sing(X)$ to mean the singular complex $\Sing(\Sigma_0,X)$. Corresponding to an $n$-simplex $\sigma$ of $\Sing(X)$, $\Delta_\sigma$ denotes the canonical $n$-simplex in $\R^{n+1}$. Define the map $\varphi_\sigma: \Delta_\sigma \rightarrow |\Cliq (X)|$ by
$$t_0 e_0 + \cdots + t_n e_n \mapsto t_0 \sigma(0) + \cdots + t_n \sigma(n).$$
For an order preserving map $f:[n] \rightarrow [m]$, one can verify the following diagram is commutative.
$$\begin{CD}
\Delta_{f^* \sigma} @>{f_*}>> \Delta_{\sigma}\\
@V{\varphi_{f^*\sigma}}VV @VV{\varphi_\sigma}V\\
|\Cliq (X)| @= |\Cliq (X)|,
\end{CD}$$
where $f_*(t_0 e_0 + \cdots + t_n e_n) = t_0 e_{f(0)} + \cdots + t_n e_{f(n)}$. Hence these $\varphi_\sigma$ induce a continuous map $\varphi_X: |\Sing(X)| \rightarrow |\Cliq(X)|$. To prove $\varphi_X$ is a homotopy equivalence, we need the following lemma.

\begin{lem}
If $X$ is a non-empty clique, then $|\Sing(X)|$ is contractible.
\end{lem} 
\begin{proof}
We note that if a homomorphism $f:X \rightarrow Y$ between $r$-sets is constant then $|\Sing (f)| : |\Sing(X)| \rightarrow |\Sing(Y)|$ is again constant. This is deduced from the fact that $\Sing(\Sigma_0)$ is a point.

Suppose that $X$ is a non-empty clique. It is clear that $|\Sing(X)|$ is connected. Let $x_0 \in V(X)$ and let $f:X \times \Sigma_1 \rightarrow X$ be the map
$$f(x,i) = \begin{cases}
x & (i=0)\\
x_0 & (i=1).
\end{cases}$$
Let $\iota_k : \Sigma_0 \rightarrow \Sigma_1$ $(k=0,1)$ be the homomorphism mapping $0$ to $k$. Then $f \circ \iota_0 = {\rm id}_X$ and $f \circ \iota_1$ is the constant homomorphism $x \mapsto x_0$ $(x \in V(X))$. Since $\Sing(X \times \Sigma_1) \cong \Sing(X) \times \Sing(\Sigma_1)$ and $\Sing(\Sigma_1)$ is connected, we have that the identity of $|\Sing (X)|$ is null-homotopic.
\end{proof}

Lemma 4.5 implies that $\varphi_X |_A:|\Sing (A)| \rightarrow |\Cliq (A)|$ is a homotopy equivalence for a finite clique $A$ which may be empty. If $A_1 ,\cdots, A_n$ be a family of finite cliques of $X$, then $A_1 \cap \cdots \cap A_n$ is also a clique. Therefore the map
$$\varphi_X |_{|\Sing(A_1)| \cap \cdots \cap |\Sing(A_n)|} : |\Sing(A_1)| \cap \cdots \cap |\Sing(A_n)| \rightarrow |\Cliq (A_1)| \cap \cdots \cap |\Cliq (A_n)|$$
is again a homotopy equivalence. By gluing homotopy equivalences (Theorem 2.4), we have that $\varphi_X$ is a homotopy equivalence. This completes the proof of Theorem 4.1.

We conclude this section by giving a few remarks. Recall that a homomorphism of $r$-sets is identified with a minimal point of $\Hom (X,Y)$ and with a vertex of $\Sing(X,Y)$. By chasing the proof carefully, one can show that the constructed homotopy equivalence preserves homomorphisms of $r$-sets.

Let $X$ be an $r$-set. The functor ${\bf Set}_r \rightarrow {\bf SSet}, Y \mapsto \Sing(X,Y)$ is a right adjoint functor by Theorem 2.3 and Theorem 3.2. Since the inclusion functor is a right adjoint functor (Proposition 3.4), the functor ${\bf Graph}_r \rightarrow {\bf SSet}$, $Y \mapsto \Sing(X,Y)$ is also a right adjoint functor. In particular, this functor preserves limits. On the other hand, $|\Hom(X,\lim Y)|$ and $|\lim \Hom(X,Y)|$ are not isomorphic but homotopy equivalent (see Proposition 3.7 in \cite{Doc}).

\section{Strong homotopy theory of $r$-sets}

Let $f,g:X \rightarrow Y$ be homomorphisms of $r$-sets. As was mentioned in Section 1, $f$ and $g$ are {\it strongly homotopic} if they belong to the same connected component of $\Hom(X,Y)$. We write $f \sim_s g$ to mean that $f$ and $g$ are strongly homotopic.

Most of results in this section are known for graphs as the $\times$-homotopy theory by Dochtermann \cite{Doc}. However, we relate the strong homotopy theory of $r$-sets to the strong homotopy theory of posets and finite simplicial complexes. For the sake of our treatment, we have that a strong homotopy equivalence $f:X \rightarrow Y$ induces strong equivalences $\Hom (Z,X) \rightarrow \Hom(Z,Y)$ and $\Hom(Y,Z) \rightarrow \Hom(X,Z)$. Furthermore, we have an alternative proof of the folding theorem (Theorem 5.6). The notion of a weak core is new. This notion is easier to handle than a core. Indeed, the associated proposition for cores of Lemma 5.14 is not trivial and is open.

Let $X$, $Y$, and $Z$ be $r$-sets. The composition map
$$* : \Hom(Y,Z) \times \Hom(X,Y) \rightarrow \Hom(X,Z)$$
is defined by
$$(\tau * \eta)(x) = \bigcup_{y \in \eta(x)} \tau (y).$$

It is easy to verify that $\tau * f = f^* (\tau)$ and $g * \sigma = g_*(\sigma)$ for homomorphisms $f:X \rightarrow Y$ and $g:Y \rightarrow Z$.

\begin{prop}
Let $X$, $Y$, and $Z$ be $r$-sets. Suppose that two homomorphisms $f,g:X \rightarrow Y$ are strongly homotopic. Then the following hold.
\begin{itemize}
\item[(1)] Poset maps $f_*,g_* :\Hom(Z,X) \rightarrow \Hom(Z,Y)$ are strongly homotopic.
\item[(2)] Poset maps $f^*,g^*:\Hom(Y,Z) \rightarrow \Hom(X,Z)$ are strongly homotopic.
\end{itemize}
\end{prop}
\begin{proof}
Let
$$\Phi : \Hom (X,Y) \rightarrow {\rm Poset}(\Hom(Z,X), \Hom(Z,Y))$$
be the adjoint of the composition map $*$. One can show that $\Phi(f) = f_*$ and $\Phi(g) = g_*$. Since $f$ and $g$ belong to the same connected component of $\Hom(X,Y)$, we have that $f_*$ and $g_*$ are strongly homotopic. The proof of (2) is similarly obtained.
\end{proof}

\begin{cor}
Let $f:X \rightarrow Y$ be a homomorphism of $r$-sets. Then the following are equivalent.
\begin{itemize}
\item[(1)] The homomorphism $f$ is a strong homotopy equivalence.
\item[(2)] For each $r$-set $Z$, the poset map $f_* : \Hom(Z,X) \rightarrow \Hom(Z,Y)$ is a homotopy equivalence.
\item[(3)] For each $r$-set $Z$, the poset map $f_* : \Hom(Z,X) \rightarrow \Hom(Z,Y)$ is a strong homotopy equivalence of posets.
\end{itemize}
A similar result holds for $f^*$.
\end{cor}
\begin{proof}
By Lemma 5.1, $(1)$ implies $(3)$. It is clear that $(3)$ implies $(2)$.

Suppose that the condition (2) holds. Since $\pi_0(f_*) :\pi_0 (\Hom(Y,X)) \rightarrow \pi_0(\Hom(Y,Y))$ is surjective, there is a homomorphism $g: Y \rightarrow X$ such that $f g \sim_s {\rm id}_Y$. On the other hand we can deduce $gf \sim_s {\rm id}_X$ from $fgf \sim_s f$ and the injectivity of $\pi_0(f_*)$. Therefore $f$ is a strong homotopy equivalence.
\end{proof}

Corresponding to a non-negative integer $n$, the $r$-set $I_n$ is defined by
$$V(I_n) = \{ 0,1,\cdots, n\},$$
$$R(I_n) = \{ (x_1,\cdots, x_r) \; | \; \textrm{There is $k \in \{ 1,\cdots, n\} $ such that } \{ x_1,\cdots,x_r\} \subset \{ k-1,k\}.\}.$$
We note that $I_1$ coincides with $\Sigma_1$.

\begin{prop}
Let $f,g:X \rightarrow Y$ be homomorphisms of $r$-sets. Then $f$ and $g$ are strongly homotopic if and only if there are a non-negative integer $n$ and a homomorphism $H: X \times I_n \rightarrow Y$ such that $H(x,0) = f(x)$ and $H(x,n) = g(x)$ for every $x \in V(X)$.
\end{prop}
\begin{proof}
Note that by Theorem 4.1 and a remark given in the end of Section 4, $f$ and $g$ are strongly homotopic if and only if $f$ and $g$ belong to the same connected component of $\Sing (X,Y)$. The proposition follows from the fact that for a pair of homomorphisms $\varphi, \psi:X \rightarrow Y$, a 1-simplex of $\Sing (X,Y)$ joining $\varphi$ to $\psi$ is the homomorphism $H:X \times I_1 \rightarrow Y$ such that $H(x,0) = \varphi (x)$ and $H(x,1) = \psi (x)$ for all $x \in V(X)$.
\end{proof}

\begin{dfn}
Let $X$ be an $r$-set. A vertex $x$ of $X$ is {\it dismantlable} if there is $y \in V(X) \setminus \{ x\}$ such that $p_i^{-1}(x) \subset p_i^{-1}(y)$ for $i \in \{ 1,\cdots, r\}$ where $p_i:R(X) \rightarrow V(X)$ is the $i$-th projection.
\end{dfn}

Let $X$ be an $r$-set. An $r$-subset of $X$ is an $r$-set $Y$ such that $V(Y) \subset V(X)$ and $R(Y) \subset R(X)$. An induced $r$-subset $Y$ of $X$ is an $r$-subset such that $R(Y) = V(Y)^r \cap R(X)$. Let $X$ be an $r$-set and let $x$ be a vertex of $X$. The induced $r$-subset of $X$ consisting of all vertices of $X$ except for $x$ is denoted by $X \setminus x$.

\begin{lem}
Let $X$ be an $r$-set and let $x$ be a vertex of $X$. If $x$ is dismantlable then the inclusion $i:X \setminus x \hookrightarrow X$ is a strong homotopy equivalence.
\end{lem}
\begin{proof}
Let $y \in V(X) \setminus \{ x\}$ such that $p_i^{-1}(x) \subset p_i^{-1}(y)$ for $i \in \{ 1,\cdots, r\}$. Let $f:V(X) \rightarrow V(X) \setminus \{ x\}$ be the map
$$f(v)=
\begin{cases}
v & (v \neq x)\\
y & (v=x).
\end{cases}
$$
Then $f$ is a homomorphism of $r$-sets and $f i$ is the identity of $X \setminus x$. Let $\eta : V(X) \rightarrow 2^{V(Y)} \setminus \{ \emptyset\}$ be the map $\eta(v) = \{ v,f(v)\}$. It is easy to verify that $\eta$ is a multi-homomorphism. Since ${\rm id}_X \leq \eta$ and $f \leq \eta$, $if$ and ${\rm id}_X$ are strongly homotopic.
\end{proof}

By Corollary 5.2 and Lemma 5.5, we have the following theorem.

\begin{thm}[Folding theorem]
Let $X$ and $Y$ be $r$-sets and let $x$ be a dismantlable vertex of $X$. We denote by $i$ the inclusion $X \setminus x \hookrightarrow X$. Then the following two maps are strong equivalences
$$i^*:\Hom(X, Y) \mapsto \Hom(X\setminus x,Y),$$
$$i_*:\Hom(Y,X \setminus x) \rightarrow \Hom(Y,X).$$
\end{thm}

\begin{rem}
Kozlov also proved that $i_*$ and $i^*$ are strong equivalences although he did not use the term.

Dochtermann pointed out that by the $\times$-homotopy theory, the folding theorem for $i^*$ yields that $i_*$ is a homotopy equivalence (Remark 6.3 of \cite{Doc}).
The folding theorem for $i^*$ is not deduced from his results of \cite{Doc} since he used it to prove (a part of) the proposition associated to Corollary 5.2.
\end{rem}

\begin{dfn}
An $r$-set is {\it stiff} if it has no dismantlable vertex.
\end{dfn}

\begin{lem}
Let $X$ be a stiff $r$-set and let $f:X \rightarrow X$ be a homomorphism. If $f \sim_s {\rm id}_X$ then $f= {\rm id}_X$.
\end{lem}
\begin{proof}
Suppose that there is a multi-homomorphism $\eta \in \Hom(X,X)$ such that $\eta > {\rm id}_X$. Let $x$ be a vertex of $X$ such that $\eta (x) \neq \{ x\}$, and let $y \in \eta(x) \setminus \{ x\}$. Let $i \in \{ 1,\cdots, r\}$ and $(x_1,\cdots, x_{i-1},x,x_{i+1},\cdots,x_r) \in R(X)$. Since
$$(x_1,\cdots,x_{i-1},y,x_{i+1},\cdots, x) \subset \eta (x_1) \times \cdots \times \eta(x) \times \cdots \times \eta(x_r) \subset R(Y).$$
Hence $x$ is dismantlable.
\end{proof}

\begin{cor}
A homomorphism $f$ between stiff $r$-sets is a strong equivalence if and only if $f$ is an isomorphism.
\end{cor}

\begin{dfn}
A homomorphism $i:X' \rightarrow X$ between $r$-sets is a {\it weak core} of $X$ if $i$ is a strong equivalence and $X'$ is stiff.

An $r$-subset $X'$ is a strong deformation retract of $X$ if there is a homomorphism
$$H:X \times I_n \rightarrow X$$
such that $H(x',i) = x'$ for each $x' \in V(X')$ and $i \in \{ 0,1,\cdots, n\}$, $H(x,0) = x$ $(x \in V(X))$, and $H(x,n) \in V(X')$ $(x \in V(X))$. The $r$-subset $X'$ is a {\it core} of $X$ if $X'$ is stiff and is a strong deformation retract of $X$.
\end{dfn}

Proposition 5.3 implies that the inclusion of a core of $X$ is a weak core. The following example shows that core and weak core are strictly different notions.

\begin{eg}
The following graph is found in Example 6.7 in \cite{Doc}.
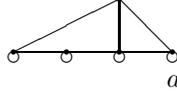
\begin{center}
\begin{picture}(100,40)(0,-5)
\put(20,10){\circle*{2}} \put(40,10){\circle*{2}} \put(60,10){\circle*{2}} \put(80,10){\circle*{2}} \put(60,30){\circle*{2}}
\put(20,10){\line(1,0){60}} \put(20,10){\line(2,1){40}} \put(60,10){\line(0,1){20}} \put(60,30){\line(1,-1){20}}
\put(20,8){\circle{4}} \put(40,8){\circle{4}} \put(60,8){\circle{4}} \put(80,8){\circle{4}}
\put(78,-3){$a$}
\end{picture}
\end{center}
Let $G$ be a graph (2-uniform hypergraph) described above. Then the map $\Sigma_0 \rightarrow G$, $0 \mapsto a$ is a weak core but is not a core. 
\end{eg}

\begin{lem}
Let $i:X' \rightarrow X$ be a weak core of $X$. Then $i$ has a retraction.
\end{lem}
\begin{proof}
Let $r :X \rightarrow X'$ be a homomorphism such that $ri \sim_s {\rm id}_{X'}$. Since $X'$ is stiff, we have that $ri$ is the identity. 
\end{proof}

By Lemma 5.5, if $V(X)$ is finite then $X$ has a core. However,there is  an $r$-set having no weak core. Indeed, there is no weak core of the 2-set $I_{\infty}$ defined by $V(I_{\infty}) = \mathbb{N}$ and $R(I_\infty) = \{ (x,y) \; | \; |x-y| \leq 1\}$.

\begin{lem}
Let $X$, $Y$ be $r$-sets which are strongly homotopy equivalent. If $X$ has a weak core, then $Y$ also has a weak core.
\end{lem}
\begin{proof}
Let $i:X' \rightarrow X$ be a weak core and let $f:X \rightarrow Y$ be a strong homotopy equivalence. By definition $fi$ is a weak core of $Y$.
\end{proof}

\begin{thm}
Let $X,Y$ be $r$-sets and let $i:X' \rightarrow X$ and $j:Y' \rightarrow Y$ be weak cores. Then $X$ and $Y$ are strong homotopy equivalent if and only if $X'$ and $Y'$ are isomorphic.
\end{thm}
\begin{proof}
Suppose that $X'$ and $Y'$ are isomorphic and let $\varphi:X' \rightarrow Y'$ be an isomorphism. Let $r:X \rightarrow X'$ be a retraction of $i:X' \rightarrow X$. Then $j \varphi r : X \rightarrow Y$ is a strong equivalence.

On the other hand, suppose that $X$ and $Y$ are strong homotopy equivalent. Let $f:X \rightarrow Y$ be a strong equivalence. Let $i:X' \rightarrow X$ and $j : Y' \rightarrow Y$ be cores, and let $r:Y \rightarrow Y'$ be a retraction of $j$. Then $rfi:X' \rightarrow Y'$ is a strong equivalence between stiff $r$-sets, and is an isomorphism.
\end{proof}

\begin{cor}
Let $X$ be an $r$-set. If $i_0:X'_0 \rightarrow X$ and $i_1: X'_1 \rightarrow X$ are weak cores, then $X_0'$ and $X_1'$ are isomorphic.
\end{cor}


\begin{thebibliography}{99}

\bibitem{BK1} E. Babson, D. N. Kozlov, {\it Complexes of graph homomorphisms}, Israel. J. Math. {\bf 152} 285-312 (2005)

\bibitem{BK2} E. Babson, D. N. Kozlov, {\it Proof of the Lov$\acute{a}$sz conjecture,} Ann. of Math. {\bf 165} (3) 965-1007 (2007)

\bibitem{Bar} J. A. Barmak, {\it Algebraic topology of finite topological spaces and applications} (Lecture Notes in Mathematics, 2032). Springer, Berlin, 2011.

\bibitem{BM} J. A. Barmak, E. G. Minian, {\it Strong homotopy types, nerves and collapses.} Discrete comput. geom. {\bf 47} (2) (2012), 301-328.


\bibitem{Doc} A. Dochtermann, {\it Hom complexes and homotopy theory in the category of graphs}, European J. Combin. {\bf 30} (2) (2009), 490-509.

\bibitem{Hov} M. Hovey, {\it Model Categories}, American Mathematical Society, Providence, R.I., 1998.

\bibitem{IK} K. Iriye, D. Kishimoto, {\it Hom complexes and hypergraph colorings}, Topology and Appl. {\bf 160} (2013), 1333-1344.

\bibitem{Koz} D. N. Kozlov, {\it Combinatorial algebraic topology}, Algorithms and Computation in Mathematics. Vol. 21 Springer, Berlin. (2008)

\bibitem{Koz2} D. N. Kozlov, {\it A simple proof for folds on both sides in complexes of graph homomorphisms}, Proc. Amer. Math. Soc. {\bf 134} no. 5 (2006) 1265-1270.

\bibitem{Lov} L. Lov$\acute{\rm a}$sz, {\it Kneser conjecture, chromatic number, and homotopy,} J. Combin. Theory Ser. A, {\bf 25} (3) 319-324 (1978)

\bibitem{May} J. P. May, {\it Simplicial objects in Algebraic Topology}, Chicago lectures in mathematics, University of Chicago, Ill. (1992), Reprint of the 1967

\bibitem{Tha} T. Thansri, {\it Simple $\Sigma_r$-homotopy types of Hom complexes and box complexes assigned to $r$-graphs}, Kyushu J. Math. {\bf 66} (2) (2012), 493-508. 

\bibitem{Qui} D. Quillen, {\it Higher algebraic K-theory I,} Lecture Notes in Mathematics {\bf 341}, Springer, 85-147 (1973)
\end{thebibliography}
\end{document}